\documentclass[11pt]{article}
\usepackage{amssymb}
\usepackage{amsthm}
\usepackage{mathrsfs}
\usepackage{mathtools}
\usepackage{float}
\usepackage{esint}
\newtheorem{thm}{Theorem}[section]
\newtheorem{cor}[thm]{Corollary}
\newtheorem{lem}[thm]{Lemma}
\newtheorem{prop}[thm]{Proposition}
\theoremstyle{definition}
\newtheorem{defn}[thm]{Definition}
\theoremstyle{remark}
\newtheorem{rem}[thm]{Remark}

\numberwithin{equation}{section}

\def\XXint#1#2#3{{\setbox0=\hbox{$#1{#2#3}{\int}$}
		\vcenter{\hbox{$#2#3$}}\kern-.5\wd0}}

\usepackage{color}

\begin{document}
	
	%
	%
	%
	%
	%
	%
	%
	%
	%
	
	
	\title{Stochastic-periodic homogenization of Poisson-Nernst-Planck equations in porous media}
	\author{ Franck Tchinda,\footnote{University of Maroua, Department of Mathematics and Computer Science, P.O. box 814, Maroua, Cameroon, email:takougoumfranckarnold@gmail.com}
	\,	Joel Fotso Tachago,\footnote{University of Bamenda, 	Higher Teachers Training College, Department of Mathematics
			P.O. Box 39
			Bambili, Cameroon, email:fotsotachago@yahoo.fr}  \\
			Joseph Dongho,\footnote{University of Maroua, Department of Mathematics and Computer Science, P.O. box 814, Maroua, Cameroon, email:joseph.dongho@fs.univ-maroua.cm}
		  \, 	and  Fridolin Tchangnwa Nya \footnote{University of Maroua, Department of Physics, P.O. box 814, Maroua, Cameroon, email:  fridolin.tchangnwa-nya@fs.univ-maroua.cm }  }
		
		\date{ } 
		
		\maketitle
	\begin{abstract} 
		
		This paper is devoted to the study of the stochastic-periodic homogenization of Poisson-Nernst-Planck equations in porous media. It is shown by the stochastic two-scale convergence method extended to periodic surfaces that  results in a global homogenized problem having suitable coefficients.
		
		
	\end{abstract}

	
	
	
	
	
	

	
	
	
	
	
	\maketitle


	\section{Introduction} \label{sect1} 
	
	We are interested in the asymptotic behavior (as $0<\varepsilon\rightarrow 0$) of the sequence of solutions $(\varpi_{\pm,\varepsilon}, \varUpsilon_{\varepsilon})$ of the  Poisson-Nernst-Planck equations (\ref{c1tc1})-(\ref{c1tc3}), with stochastic-deterministic parameters in porous periodic media :
	\begin{equation}\label{c1tc1}
		\dfrac{\partial \varpi_{\pm,\varepsilon}}{\partial t} - D_{\pm} \textup{div} \left(\nabla\varpi_{\pm,\varepsilon} + \dfrac{F}{R\varTheta}z_{\pm}\varpi_{\pm}\nabla\varUpsilon_{\varepsilon}\right) = 0 \quad \textup{in} \; Q_{\varepsilon}^{f}\times\Lambda,
	\end{equation}
	
	\begin{equation}\label{c1tc2}
		\textup{div}\left(\varrho_{f}\left(\mathcal{T}\left(\frac{x}{\varepsilon}\right)\omega,\frac{x}{\varepsilon^{2}}\right)\nabla\varUpsilon_{\varepsilon}\right) = -F \left(z_{+}\varpi_{+,\varepsilon} - z_{-}\varpi_{-,\varepsilon}\right) \;\; \textup{in} \; Q_{\varepsilon}^{f}\times\Lambda,
	\end{equation}
	
	\begin{equation}\label{c1tc3}
		\textup{div}\left(\varrho_{s}\left(\mathcal{T}\left(\frac{x}{\varepsilon}\right)\omega,\frac{x}{\varepsilon^{2}}\right)\nabla\varUpsilon_{\varepsilon}\right) = 0 \quad \textup{in} \; Q_{\varepsilon}^{s}\times\Lambda.
	\end{equation}
Let us specify data in (\ref{c1tc1})-(\ref{c1tc3}) : 
\begin{itemize}
	\item $Q_{\varepsilon}^{f}=]0,T[\times \Omega_{\varepsilon}^{f}$ and $Q_{\varepsilon}^{s}=]0,T[\times \Omega_{\varepsilon}^{s}$,
	\item  $\Omega_{\varepsilon}^{f}$ and $\Omega_{\varepsilon}^{s}$ are the  porous media (see Subsection \ref{c1sect11} for a total description) and $T$ is a fixed positive real,
	\item  $\Lambda$ is a support of a probability space and $\mathcal{T}$ is an $N$-dimensional dynamical system on $\Lambda$ with integer $N=2$ or $3$,
	\item  $0<\varepsilon\leq 1$ is a parameter controlling the microvariation in $Q_{\varepsilon}^{f}\times\Lambda$ (or $Q_{\varepsilon}^{s}\times\Lambda$),
	\item  $\textup{div}\equiv \textup{div}_{x}$ is the usual divergence in $\mathbb{R}_{x}^{N}$ of variables $x=(x_{1},\cdots,x_{N})$, 
	\item  (\ref{c1tc1}) is the Nernst-Planck equation : $\varpi_{\pm,\varepsilon}$ denotes the concentration of the cations (resp. anions) of valence $z_{\pm}$, $D_{\pm}$ their molecular diffusion coefficient, $\varUpsilon_{\varepsilon}$ denotes the electrical potential in the solution, $F$ and $R$ denote respectively the Faraday and the ideal gas constant, and $\varTheta$ the temperature of the fluid, $\varrho_{f}$ (resp. $\varrho_{s}$) denote the dielectric constant of the fluid phase $\Omega_{\varepsilon}^{f}$ (resp. of the solid phase $\Omega_{\varepsilon}^{s}$),
	\item  (\ref{c1tc2})-(\ref{c1tc3}) is the Poisson equation characterizing the electrical potential $\varUpsilon_{\varepsilon}$ and result from Maxwell's equations.
\end{itemize}
To these equations, we will associate the boundary Neumann and initial Cauchy conditions (see (\ref{c2tc6})-(\ref{c2tc9}) and (\ref{c2tc12}) ). In general, the Poisson-Nernst-Planck  system  is a well-established framework for describing the movement of ions and the variations of the electric field in an electrolyte (see e.g. \cite{kirby1}).
In biology, examples of such phenomena include ion transportation in biological tissues, such as ion transfer through ion channels of cell membrane \cite[Chapter 3]{kener1}  and application to neuronal signal propagation \cite{pod1} as well as understanding disease characteristics \cite{borys1}. In
geology, these equations on complex media appear when modeling electro-kinetic flow through porous rocks \cite{aliza1}
and, in engineering sciences, when modeling electro-osmosis in porous media \cite{wang1}.
 \par The homogenization of the Poisson-Nernst-Planck type equations have been studied by many authors. In \cite{gagneu1}, G. Gagneux and O. Millet studied the periodic homogenization of Poisson-Nernst-Planck equations of form (\ref{c1tc1})-(\ref{c1tc3}) where only the deterministic aspect is taken into account and in the same way, N. Ray, A. Muntean and P. Knabner, in \cite{ray1} study the periodic homogenization  of a Stokes-Nernst-Planck-Poisson system. In \cite{smuc1}, M. Schmuck and M.Z. Bazant studied the periodic homogenization of Poisson-Nernst-Planck equations via the asymptotic multiscale expansion method. Other works on the  numerical analysis and convergent discretizations for the nonstationary Nernst-Planck-Poisson system are detailed in \cite{proh1,proh2}. 
 \par  The idea in this paper is to start from the work of G. Gagneux and O. Millet \cite{gagneu1} by integrating the random variable, and to homogenize the new system that we obtain, using the stochastic two-scale convergence defined
by M. Sango et J.L. Woukeng in \cite{sango}. Note that this convergence generalizes both the concept of stochastic two-scale convergence in the mean of A. Bourgeat et al. \cite{bourgeat} and the concept of two-scale convergence of G. Nguetseng \cite{nguet1}. The choice of such convergence is motivated by the fact that in nature almost all phenomena behave randomly in some scales and deterministically in some other scales (see e.g., \cite{sango,sango2}). But, since the boundary Neumann conditions (\ref{c2tc6})-(\ref{c2tc9}) are defined in periodic surfaces, we have need to extends firstly  the concept of stochastic two-scale convergence to the periodic surfaces. 
\par The paper is divided into sections each revolving around a specific aspect. In section \ref{c1sect1} we start by describing the stochastic Poisson-Nernst-Planck (s-PNP) problem in porous periodic medium. Section \ref{sect2} dwells on some preliminary results on stochastic two-scale convergence method that we extend to periodic surfaces. To end, in the Section \ref{sect3}, we use this method to establishe the global homogenized problem resulting from the stochastic Poisson-Nernst-Planck equations.

	\section{Description of s-PNP problem in porous medium}\label{c1sect1}
	
	In this section, we start by describing the geometry of porous domain, and then the stochastic Poisson-Nernst-Planck problem in porous media, which for the readers’ convenience
	is not very different from the one adopted in \cite{gagneu1}. But we will assume that the data and the functions depend also of a stochastic parameter. To end, we show the existence and uniqueness of solutions of s-PNP microscopic problem.
	
	\subsection{Geometry of domain}\label{c1sect11}
	In order to describe the geometry of the porous domain, we recall some notations that we will use in the sequel.
	\begin{itemize}
		\item $\Omega$ is a bounded connected open set of $\mathbb{R}^{N}$ (integer $N>1$) with a Lipschitz boundary $\partial\Omega$.
		\item $V=H^{1}(\Omega)$ and $H=L^{1}(\Omega)$ are Hilbert spaces with the scalar products, the embedding of $V$ in $H$ being dense and compact.
		\item  we identify the pivot space $H$ with its topologic dual, such as $V'$, the dual of $V$, can be identified with an \textit{over-space} of $H$.
		\item We note $\langle\cdot,\cdot\rangle$ the duality bracket between $V$ and $V'$. For a couple of regular functions, this duality bracket coincides with the usual scalar product of $L^{2}(\Omega)$.
		\item 	For parabolic time depending problems, its is also classical to introduce the space (see e.g., \cite{jlion1,jlion2}) 
		\begin{equation*}
			\left\{u\in L^{2}\left([0,T]; H^{1}(\Omega)\right), \; \dfrac{\partial u}{\partial t} \in L^{2}\left([0,T]; L^{2}(\Omega)\right)\right\}
		\end{equation*}
		which is a Hilbert space for its natural topology that can be identified with $H^{1}(Q)$, with $Q=]0,T[\times\Omega$.  Moreover, the injection of $H^{1}(Q)$ in $L^{2}\left([0,T]; H\right)$ (identified with $L^{2}(Q)$) is compact \cite[Page 57]{jlion2}, and the injection of $H^{1}(Q)$ in $\mathcal{C}^{0}\left([0,T]; H\right)$ is continuous provided an appropriate choice of a representative element.
 	 \item  We denote by $Y$ the unit reference cube $[0,1]^{N}$, $Y_{s}$ an open set of $Y$ representing the solid phase, and $Y_{f}=Y\backslash \overline{Y_{s}}$ its complementary (assumed to be connected) representing the fluid phase.
 	 \item  The Lipschitzian boundary $\Gamma= \Gamma^{sf}\cup\Gamma^{ss}$ is composed of the internal boundary $\Gamma^{sf}$ between the fluid and the solid phases and of the boundary $\Gamma^{ss}$ between the solid parts of two adjacent unit cubes constituting the microstructure of the material. Equivalently, we denote by $\Gamma^{ff}$ the boundary between the fluid phases of two adjacent unit cubes.
 	 \item  Finally, in the sequel, we will still denote by $Y$, $Y_{s}$, $Y_{f}$ and $\Gamma^{sf}$ the extensions of $Y$, $Y_{s}$, $Y_{f}$ and $\Gamma^{sf}$ to $\mathbb{R}^{N}$ by $Y$-periodicity. 
 \end{itemize}
\par 	Introducing now a sequence $\varepsilon$ of positive reals\footnote{it is called a \textit{fundamental sequence} (see e.g. \cite{sango}) } converging to 0, we define the whole fluid phase where the fluid can circulate 
	
	\begin{equation}
		\Omega_{\varepsilon}^{f} = \left\{x\in\Omega, \; \dfrac{x}{\varepsilon} \in Y_{f}\right\}
	\end{equation} 
	and its complementary representing the whole solid phase 
	
	\begin{equation}
		\Omega_{\varepsilon}^{s} = \left\{x\in\Omega, \; \dfrac{x}{\varepsilon} \in Y_{s}\right\}
	\end{equation} 
	of the microstructure of the material. We denote by $\Gamma_{\varepsilon}^{sf}$ the inner boundary of $\Omega$ between the fluid and the solid phases, periodic surface of dimension $(N-1)$ such as :
	
		\begin{equation}
		\Gamma_{\varepsilon}^{sf} = \left\{x\in\Omega, \; \dfrac{x}{\varepsilon} \in \Gamma^{sf}\right\}.
	\end{equation} 
	See Figure \ref{domain1} below for a total schematization of porous periodic medium.
	
	    	\begin{figure}[H]
	    	\centering
	    	\includegraphics[width=0.9\linewidth]{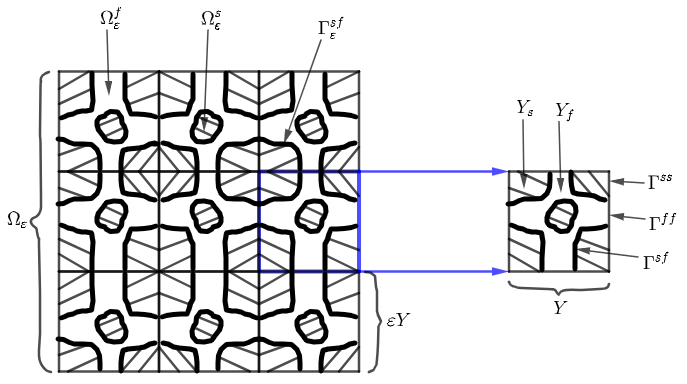}
	    	\caption{A 2D representation of  porous periodic medium $\Omega_{\varepsilon}$ and its unit cell $Y$}
	    	\label{domain1}
	    \end{figure}
	    
	    In the sequel we will denote $Q_{\varepsilon}^{f} = ]0,T[\times \Omega_{\varepsilon}^{f}$ and $V_{\varepsilon} = H^{1}(\Omega_{\varepsilon}^{f})$, $H^{1}(Q_{\varepsilon}^{f})$, $L^{\infty}(Q_{\varepsilon}^{f})$, etc., the expressions of the functional spaces which now depend on $\varepsilon$.
	
	\subsection{Microscopic s-PNP problem}
	
	The Poisson-Nernst-Planck (PNP) system of equations is a well-established framework for describing the movement of ions and the variations of the electric field in an electrolyte (see e.g. \cite{kirby1}). In order to describe the stochastic Poisson-Nernst-Planck equations, we add a random variable to unknown functions. 
	Thus, considering the porous medium as defined above (see Figure \ref{domain1}), the diffusion of ions in the fluid phase is described by Nernst-Planck equation coupled with the Poisson equation characterizing the electrical field accelerating or slowing down the ionic transfer. Nernst-Planck equation is generally written for an isotropic diffusivity assumed in the fluid phase (see, e.g. \cite{bard1,bil1,mil1,lipt1}) 
	
	   \begin{equation}\label{c2tc1}
	   	 \dfrac{\partial \varpi_{\pm,\varepsilon}}{\partial t} - D_{\pm} \textup{div} \left(\nabla\varpi_{\pm,\varepsilon} + \dfrac{F}{R\varTheta}z_{\pm}\varpi_{\pm}\nabla\varUpsilon_{\varepsilon}\right) = 0 \quad \textup{in} \; Q_{\varepsilon}^{f}\times\Lambda
	   \end{equation}
	   where $\varpi_{\pm,\varepsilon}$ denotes the concentration of the cations (resp. anions) of valence $z_{\pm}$, $D_{\pm}$ their molecular diffusion coefficient, $\varUpsilon_{\varepsilon}$ denotes the electrical potential in the solution, $F$ and $R$ denote respectively the Faraday and the ideal gas constant, and $\varTheta$ the temperature of the fluid. 
	   Note that equation (\ref{c2tc1}) corresponds to the mass conservation for the ionic species of concentration $\varpi_{\pm,\varepsilon}$ diffusing in the liquid phase.  
	   \par  The Poisson equation characterizing the electrical potential $\varUpsilon_{\varepsilon}$ in the fluid phase and in the solid phase is stated as 
	   
	   \begin{equation}\label{c2tc2}
	   	\textup{div}\left(\varrho_{f}\left(\mathcal{T}\left(\frac{x}{\varepsilon}\right)\omega,\frac{x}{\varepsilon^{2}}\right)\nabla\varUpsilon_{\varepsilon}\right) = -F \left(z_{+}\varpi_{+,\varepsilon} - z_{-}\varpi_{-,\varepsilon}\right) \quad \textup{in} \; Q_{\varepsilon}^{f}\times\Lambda
	   \end{equation}
	   where the quantity $\rho^{\star}= F \left(z_{+}\varpi_{+,\varepsilon} - z_{-}\varpi_{-,\varepsilon}\right)$ denotes the total electrical charge density in the fluid phase. \\
	   In the solid phase, the Poisson equation reduces to
	   
	   \begin{equation}\label{c2tc3}
	   	\textup{div}\left(\varrho_{s}\left(\mathcal{T}\left(\frac{x}{\varepsilon}\right)\omega,\frac{x}{\varepsilon^{2}}\right)\nabla\varUpsilon_{\varepsilon}\right) = 0 \quad \textup{in} \; Q_{\varepsilon}^{s}\times\Lambda
	   \end{equation}
	   as the concentrations of ionic species are only defined in the fluid phase. Note that $\varrho_{f}$ (resp. $\varrho_{s}$) denotes the dielectric constant of the fluid phase $\Omega_{\varepsilon}^{f}$ (resp. of the solid phase $\Omega_{\varepsilon}^{s}$) and also depends on the random variable. \\
	   Introducing the function 
	   \begin{equation}\label{c2tc4}
	   	\vartheta := \varrho_{f}\chi_{Y_{f}} + \varrho_{s}\chi_{Y_{s}} 
	   \end{equation}
	   representing the dielectric values on the whole domain $\Omega\times \Lambda$ through the characteristic phase functions 
	$\chi_{Y_{f}}$ and $\chi_{Y_{s}}$, (\ref{c2tc2}) and (\ref{c2tc3}) can be written in the global formulation in the sense of distributions as 
	
	 \begin{equation}\label{c2tc5}
	 	\begin{array}{rl}
		\textup{div}\left(\vartheta\left(\mathcal{T}\left(\frac{x}{\varepsilon}\right)\omega,\frac{x}{\varepsilon^{2}}\right)\nabla\varUpsilon_{\varepsilon}\right) = & -F \left(z_{+}\varpi_{+,\varepsilon} - z_{-}\varpi_{-,\varepsilon}\right)\chi_{\Omega_{\varepsilon}^{f}}  \\
		 & + g(\varepsilon,t,x,\omega)\delta_{(\Gamma_{\varepsilon}^{sf})} \qquad \quad \quad \;\; \textup{in} \; Q\times\Lambda 
	\end{array}
	\end{equation}
	where the last term denotes a measure, with support on $\Gamma_{\varepsilon}^{sf}$, governing the jump of the electrical flux through the interface $\Gamma_{\varepsilon}^{sf}$. This term will be explained later by the assertions (\ref{c2tc8})-(\ref{c2tc9}) and the Grahame relation.
	
	To (\ref{c2tc1})-(\ref{c2tc5}), let us write the associated Neumann boundary conditions  with the Nernst-Planck equation and the system is assumed to be insulated with respect to the exterior. Indeed, for all $t>0$ and almost all $\omega\in\Lambda$, we set 
	
	\begin{equation}\label{c2tc6}
		\begin{array}{r}
		  D_{\pm} \left(\nabla\varpi_{\pm,\varepsilon}(t,\cdot,\omega) + \dfrac{F}{R\varTheta}z_{\pm}\varpi_{\pm}(t,\cdot,\omega)\nabla\varUpsilon_{\varepsilon}(t,\cdot,\omega)\right)\cdot \vec{n} = 0, \\
		   \textup{on} \; \Gamma_{\varepsilon}^{sf}\cup \Gamma_{\varepsilon}^{f_{\textup{ext}}} 
		\end{array}
	\end{equation}
	
	\begin{equation}\label{c2tc7}
		\dfrac{\partial \varUpsilon_{\varepsilon}(t,\cdot,\omega)}{\partial \vec{n}} = 0,
		 \quad \textup{on} \; \Gamma_{\varepsilon}^{f_{\textup{ext}}}\cup \Gamma_{\varepsilon}^{s_{\textup{ext}}} 
	\end{equation}
	where $\Gamma_{\varepsilon}^{f_{\textup{ext}}} = \partial\Omega\cap\Omega_{\varepsilon}^{f}$ (resp. $\Gamma_{\varepsilon}^{s_{\textup{ext}}} = \partial\Omega\cap\Omega_{\varepsilon}^{s}$) denotes the external fluid boundary to $\Omega_{\varepsilon}^{f}$ (resp. to the external solid boundary to $\Omega_{\varepsilon}^{s}$) and $\vec{n}$ denotes the external unit normal to $\Omega_{\varepsilon}^{f}$. \\
	Moreover, the electrical double layer phenomenon at the interface $\Gamma_{\varepsilon}^{sf}$ will be modelled through the boundary condition involving the surface charge density $\kappa$. In the general case, accounting for the jump of the electrical flux through the interface $\Gamma_{\varepsilon}^{sf}$, noted $\left[\left[\vartheta\left(\mathcal{T}\left(\frac{x}{\varepsilon}\right)\omega,\frac{x}{\varepsilon^{2}}\right)\nabla\varUpsilon_{\varepsilon}(t,x,\omega)\right]\right]\cdot \vec{n}$, we have 
	
	\begin{equation}\label{c2tc8}
		\left[\left[\vartheta\left(\mathcal{T}\left(\frac{x}{\varepsilon}\right)\omega,\frac{x}{\varepsilon^{2}}\right)\nabla\varUpsilon_{\varepsilon}(t,\cdot,\omega)\right]\right]\cdot \vec{n} = \varepsilon \kappa \left(\varUpsilon_{\varepsilon}(t,\cdot,\omega)\right) \quad \textup{on} \; \Gamma_{\varepsilon}^{sf}
	\end{equation}
	where $\kappa$ is defined for all $t>0$ and almost all $\omega\in\Lambda$ by the Grahame relation (see \cite{gra1,gra2})
	
	\begin{equation}\label{c2tc9}
		 \kappa \left(\varUpsilon_{\varepsilon}(t,\cdot,\omega)\right) := -\eta\left(\mathcal{T}\left(\frac{x}{\varepsilon}\right)\omega,\frac{x}{\varepsilon^{2}}\right) \gamma\left(\varUpsilon_{\varepsilon}(t,\cdot,\omega)\right) \quad \textup{on} \; \Gamma_{\varepsilon}^{sf}.
	\end{equation}
	Here for almost all $\omega\in\Lambda$, $\eta(\omega,\cdot)$  denotes the dielectric permittivity coefficient at the interface $\Gamma^{sf}$, possibly heterogeneous, extended by $Y$-periodicity, verifying $\eta \in L^{\infty}(\Lambda\times\Gamma^{sf})$, $\eta(\omega,\cdot) \geq \eta_{0}(\omega,\cdot)>0$ for almost all $\omega\in\Lambda$. The function $\gamma$ must be a monotonic continuously differentiable Lipschitz-function satisfying the following conditions 
	
	\begin{equation}
		\gamma(0) = 0, \qquad \gamma^{\prime}(r) \geq \alpha, \quad \textup{with} \;\; \alpha>0, \;\, r\in\mathbb{R}.
	\end{equation}
	In the sequel we assume that the function $\vartheta$ belongs to $\mathcal{C}^{\infty}(\Lambda;\mathcal{C}_{per}(Y))$.

	\subsection{Existence and uniqueness of a weak solution for (\ref{c2tc1})-(\ref{c2tc8})}
	
	The aim of this section is to define the weak formulation of s-PNP system and to show the existence and uniqueness of its solution. For that, we argue as in \cite[Page 74]{gagneu1}.
	\par Multiplying the system of equations (\ref{c2tc1})-(\ref{c2tc8}) with the test functions $ \phi\in H^{1}(\Omega_{\varepsilon}^{f})$, $ \psi\in H^{1}(\Omega)$ and integrating by parts we get, $\forall\, t>0$ and for $\mu$-a.e. $\omega\in\Lambda$, the following weak formulation of problem to which we add the initial Cauchy conditions (\ref{c2tc12}) :
	
	\begin{equation}\label{c2tc10}
		\begin{array}{l}
		\int_{\Omega_{\varepsilon}^{f}} \dfrac{\partial \varpi_{\pm,\varepsilon}(t,\cdot,\omega)}{\partial t} \phi dx + D_{\pm} \int_{\Omega_{\varepsilon}^{f}} \nabla\varpi_{\pm,\varepsilon}(t,\cdot,\omega).\nabla\phi dx \\
		
		\\
		\quad + D_{\pm}\dfrac{F}{R\varTheta} \int_{\Omega_{\varepsilon}^{f}} z_{\pm}\varpi_{\pm,\varepsilon}(t,\cdot,\omega)\nabla\varUpsilon_{\varepsilon}(t,\cdot,\omega).\nabla\phi dx = 0,
	\end{array}
	\end{equation}
	\vspace{0.5cm}
	\begin{equation}\label{c2tc11}
		\begin{array}{l}
			\int_{\Omega} \vartheta\left(\mathcal{T}\left(\frac{\cdot}{\varepsilon}\right)\omega,\frac{\cdot}{\varepsilon^{2}}\right)\nabla\varUpsilon_{\varepsilon}(t,\cdot,\omega).\nabla\psi dx  \\
			
			\\
			\quad	+ \varepsilon \int_{\Gamma^{sf}_{\varepsilon}} \eta\left(\mathcal{T}\left(\frac{\cdot}{\varepsilon}\right)\omega,\frac{\cdot}{\varepsilon^{2}}\right) \gamma\left(\varUpsilon_{\varepsilon}(t,\cdot,\omega)\right) \psi dS_{\varepsilon}  \\
			
		 \\
		\quad	=
			 F \int_{\Omega_{\varepsilon}^{f}} \left(z_{+}\varpi_{+,\varepsilon}(t,\cdot,\omega) - z_{-}\varpi_{-,\varepsilon}(t,\cdot,\omega)\right)\psi dx, 
		\end{array}
	\end{equation}
	\vspace{0.5cm}
	\begin{equation}\label{c2tc12}
	\varpi_{\pm,\varepsilon}(0,\cdot,\omega) = \varpi^{0}_{\pm,\varepsilon}(\cdot,\omega) \quad \textup{in} \;\, \Omega_{\varepsilon}.	
	\end{equation}
	\begin{defn}
		We call $(\varpi_{\pm,\varepsilon})= (\varpi_{+,\varepsilon}, \varpi_{-,\varepsilon})$ and $\varUpsilon_{\varepsilon}$ the weak solutions of s-PNP equations (\ref{c2tc1})-(\ref{c2tc8}) if
		 $$(\varpi_{\pm,\varepsilon})= (\varpi_{+,\varepsilon}, \varpi_{-,\varepsilon}) \in L^{\infty}\left([0,T]; L^{2}\left(\Lambda;H^{1}(\Omega_{\varepsilon}^{f})\right)\right)^{2} \cap L^{2}\left(\Lambda;H^{1}(Q_{\varepsilon}^{f})\right)^{2}$$
		  and 
		  $$\varUpsilon_{\varepsilon} \in L^{\infty}\left([0,T]; L^{2}\left(\Lambda;H^{1}(\Omega)\right)\right)$$
		   and equations (\ref{c2tc10})-(\ref{c2tc12}) are satisfied for all test functions $ \phi\in H^{1}(\Omega_{\varepsilon}^{f})$, $ \psi\in H^{1}(\Omega)$. 
	\end{defn}
	\begin{prop}\label{c2prop1}
		Assume that $\Omega$ is smooth enough and initial conditions $(\varpi^{0}_{\pm,\varepsilon}) = \left(\varpi^{0}_{+,\varepsilon}, \varpi^{0}_{-,\varepsilon}\right) \in L^{2}\left(\Lambda;L^{\infty}(\Omega_{\varepsilon}^{f})\right)^{2} \cap L^{2}\left(\Lambda;H^{1}(\Omega_{\varepsilon}^{f})\right)^{2}$ with $(\varpi^{0}_{\pm,\varepsilon}(\cdot,\omega)) \geq 0$ for almost all $\omega\in\Lambda$. Then for $\mu$-a.e. $\omega\in\Lambda$, the variational problem (\ref{c2tc10})-(\ref{c2tc12}) admits a unique solution $(\varpi_{\pm,\varepsilon}(\cdot,\cdot,\omega), \varUpsilon_{\varepsilon}(\cdot,\cdot,\omega))$ such that $(\varpi_{\pm,\varepsilon}(\cdot,\cdot,\omega)) = \left(\varpi_{+,\varepsilon}(\cdot,\cdot,\omega), \varpi_{-,\varepsilon}(\cdot,\cdot,\omega)\right) \in L^{\infty}\left([0,T]; H^{1}(\Omega_{\varepsilon}^{f})\right)^{2} \cap H^{1}(Q_{\varepsilon}^{f})^{2}$ and $\varUpsilon_{\varepsilon}(\cdot,\cdot,\omega) \in \mathcal{C}^{0}\left([0,T]; H^{1}(\Omega)\right)$.
	\end{prop}
	\begin{proof}
		The proof derived from \cite[Proposition 1]{gagneu1}  using the same arguments as in \cite[Page 42]{sango2} or \cite[Page 48]{bourgeat} 
	\end{proof}
	\begin{rem}
	As in \cite[Page 42]{sango2}, we deduce from the Proposition \ref{c2prop1},  that the unique solution $(\varpi_{\pm,\varepsilon}, \varUpsilon_{\varepsilon})$ of weak problem (\ref{c2tc10})-(\ref{c2tc12}) is such that $(\varpi_{\pm,\varepsilon}) = \left(\varpi_{+,\varepsilon}, \varpi_{-,\varepsilon}\right) \in L^{\infty}\left([0,T];L^{2}(\Lambda; H^{1}(\Omega_{\varepsilon}^{f}))\right)^{2} \cap L^{2}\left(\Lambda;H^{1}(Q_{\varepsilon}^{f})\right)^{2}$ and $\varUpsilon_{\varepsilon} \in \mathcal{C}^{0}\left([0,T]; L^{2}(\Lambda; H^{1}(\Omega))\right)$.
	\end{rem}
	
	We give now some physical properties about the concentrations $\varpi_{\pm,\varepsilon}$ and the electrical potential $\varUpsilon_{\varepsilon}$. This is a consequence of Proposition \ref{c2prop1}. Thus for all $t>0$ and for $\mu$-a.e. $\omega\in\Lambda$ we have : 
	\begin{equation}\label{c2tc13}
		(\varpi_{\pm,\varepsilon}(\cdot,\cdot,\omega)) \in L^{\infty}(Q_{\varepsilon}^{f})^{2}, \quad \varpi_{\pm,\varepsilon}(t,\cdot,\omega) \geq 0 \quad \textup{in} \;\, \Omega_{\varepsilon}^{f},
	\end{equation}
	
	\begin{equation}\label{c2tc14}
		\int_{\Omega_{\varepsilon}^{f}} \varpi_{\pm,\varepsilon}(t,\cdot,\omega) dx = \int_{\Omega_{\varepsilon}^{f}} \varpi^{0}_{\pm,\varepsilon}(\cdot,\omega) dx,
	\end{equation}
	as consequence of (\ref{c2tc10}) with the test function $\phi = 1_{\Omega_{\varepsilon}^{f}}$,
	\begin{equation}\label{c2tc15}
		\begin{array}{l}
			F \int_{\Omega_{\varepsilon}^{f}} \left(z_{+}\varpi_{+,\varepsilon}(t,\cdot,\omega) - z_{-}\varpi_{-,\varepsilon}(t,\cdot,\omega)\right) dx \\
	   	= \varepsilon \int_{\Gamma^{sf}_{\varepsilon}} \eta\left(\mathcal{T}\left(\frac{\cdot}{\varepsilon}\right)\omega,\frac{\cdot}{\varepsilon^{2}}\right) \gamma\left(\varUpsilon_{\varepsilon}(t,\cdot,\omega)\right) dS_{\varepsilon},  
	\end{array}
	\end{equation}
	as consequence of (\ref{c2tc11}) with the test function $\psi = 1_{\Omega}$. Therefore, (\ref{c2tc14}) implies that the quantity
	\begin{equation*}
	\varepsilon \int_{\Gamma^{sf}_{\varepsilon}} \vartheta\left(\varUpsilon_{\varepsilon}(t,\cdot,\omega)\right) dS_{\varepsilon} = -\varepsilon \int_{\Gamma^{sf}_{\varepsilon}} \eta\left(\mathcal{T}\left(\frac{\cdot}{\varepsilon}\right)\omega,\frac{\cdot}{\varepsilon^{2}}\right) \gamma\left(\varUpsilon_{\varepsilon}(t,\cdot,\omega)\right) dS_{\varepsilon}	
	\end{equation*}
	is constant with respect to time $t$ and we have for all $t>0$ and for $\mu$-a.e. $\omega\in\Lambda$ : 
		\begin{equation}\label{c2tc16}
		\varepsilon \int_{\Gamma^{sf}_{\varepsilon}} \vartheta\left(\varUpsilon_{\varepsilon}(t,\cdot,\omega)\right) dS_{\varepsilon} = 	-Fz_{+} \int_{\Omega_{\varepsilon}^{f}} \varpi^{0}_{+,\varepsilon}(\cdot,\omega) dx  + Fz_{-} \int_{\Omega_{\varepsilon}^{f}} \varpi^{0}_{-,\varepsilon}(\cdot,\omega) dx
	\end{equation}

	
	\section{Fundamentals of stochastic two-scale convergence}\label{sect2}
	
	\subsection{Some notations}
	Throughout this section, all the vector spaces are assumed to be real vector spaces, and the scalar functions are assumed real values.
	We refer to  \cite{bourgeat}, \cite{sango}, \cite{sango2} for notations. 
	\begin{itemize}
		\item $\Omega$ is an open bounded set of $\mathbb{R}^{N}$, integer $N>1$ and
	 $|\Omega|$ denotes the volume of $\Omega$ with respect to the Lebesgue mesure.
\item	$Y=[0,1]^{N}$ denote the unit cube of $\mathbb{R}^{N}$.
\item  $Q= ]0,T[\times \Omega$ where $T>0$ be a real number.
\item	$E=(\varepsilon_{n})_{n\in\mathbb{N}}$ is a \textit{fundamental sequence}. 
\item	$\mathcal{D}(\Omega)=\mathcal{C}^{\infty}_{0}(\Omega)$ is the vector space of smooth functions with compact support in $\Omega$. 
\item	$\mathcal{C}^{\infty}(\Omega)$ is the vector space of smooth functions on $\Omega$. 
\item	$\mathcal{K}(\Omega) $ is the vector space of continuous functions with compact support in $\Omega$. 
\item	$L^{p}(\Omega)$, integer $p \in [1,+\infty]$, is Lebesgue space of functions on $\Omega$.  
\item	$W^{1,p}(\Omega) = \{ v \in L^{p}(\Omega) \, : \, \frac{\partial v}{\partial x_{i}} \in L^{p}(\Omega), \; 1 \leq i \leq N \}$, where derivatives are taken in the weak sense, is classical Sobolev's space of functions and $W^{1,2}(\Omega) \equiv H^{1}(\Omega)$. 
\item	$W^{1,p}_{0}(\Omega)$ the set of functions in $W^{1,p}(\Omega)$ with zero boundary condition and $W^{1,2}_{0}(\Omega) \equiv H^{1}_{0}(\Omega)$.
\item	$\nabla$ or $D$ (resp. $\nabla_{y}$ or $D_{y}$) denote the (classical) gradient operator on $\Omega$ (resp. on $Y$). 
\item	$\textup{div}$ (resp. $\textup{div}_{y}$) the (classical) divergence operator on $\Omega$ (resp. on $Y$).
\item	If $F(\mathbb{R}^{N})$ is a given function space, we denote by $F_{per}(Y)$ the space of functions in $F_{loc}(\mathbb{R}^{N})$ that are $Y$-periodic, and by $F_{\#}(Y)$ those functions in $F_{per}(Y)$ with mean value zero. 
\item	$(\Lambda, \mathscr{M}, \mu)$ : measure space with probability measure $\mu$. 
\item	$\{\mathcal{T}(y): \Lambda \rightarrow \Lambda \, , \, y \in \mathbb{R}^{N} \}$ is an $N$-dimensional dynamical system on $\Omega$  
\item	$L^{p}(\Lambda)$, integer $p \in [1,+\infty]$, is Lebesgue space of functions on $\Lambda$.  
\item	$L^{p}_{nv}(\Lambda)$ is the set of all $\mathcal{T}$-invariant functions in $L^{p}(\Lambda)$.
\item	$D_{i}^{\omega} : L^{p}(\Lambda)\rightarrow L^{p}(\Lambda)$, $1 \leq i \leq N$, $\omega\in\Lambda$ is the $i$-th stochastic derivative operator. 
\item	$D_{\omega} = (D_{1}^{\omega}, \cdots, D_{N}^{\omega})$ is the stochastic gradient operator on $\Lambda$.
\item	$W^{1,p}(\Lambda) = \left\{ f \in L^{p}(\Lambda) : D^{\omega}_{i}f \in L^{p}(\Lambda) \, (1 \leq i \leq N) \right\}$,
	where $D^{\omega}_{i}f$ is taken in the distributional sense on $\Lambda$ and $W^{1,2}(\Lambda) \equiv H^{1}(\Lambda)$.
\item	$W_{\#}^{1,p}(\Lambda)$ is the separated completion of $\mathcal{C}^{\infty}(\Lambda)$ in $L^{p}(\Lambda)$  with respect to the norm $\|f\|_{\#,p} = \left( \sum_{i=1}^{N} \|\overline{D}^{\omega}_{i}f\|^{p}_{L^{p}(\Lambda)} \right)^{\frac{1}{p}}$, where $\overline{D}_{i}^{\omega}$ is the extension of $D_{i}^{\omega}$ to $W_{\#}^{1,p}(\Lambda)$.
\item  $\overline{D}_{\omega}$ is the extension of $D_{\omega}$ to $W_{\#}^{1,p}(\Lambda)$.
\item	the operator $\textup{div}_{\omega} : L^{p'}(\Omega)^{N} \rightarrow (W_{\#}^{1,p}(\Omega))'$, $(p' = \frac{p}{p-1})$ defined by 
\begin{equation*}
	\langle \textup{div}_{\omega}u, v \rangle = - \langle u, \overline{D}_{\omega}v \rangle =  - \sum_{i=1}^{N} \int_{\Omega} u_{i} \overline{D}_{i}^{\omega}v \,d\mu
\end{equation*}
$(u = (u_{i}) \in L^{p'}(\Omega)^{N}, v \in W_{\#}^{1,p}(\Omega))$ naturally extends the stochastic divergence operator in $\mathcal{C}^{\infty}(\Omega)$.  
	\end{itemize}

\subsection{Stochastic two-scale convergence in $L^{p}(\Omega\times\Lambda)$}	
 Now we recall some concepts about the stochastic two-scale convergence which is the generalization of both two-scale convergence in the mean (of Bourgeat and al. \cite{bourgeat}) and two-scale convergence (of Nguetseng \cite{nguet1}). Let $F(\mathbb{R}^{N})$ be a given function space. We also assume that $L^{p}(\Lambda)$ is separable. We denote by $F_{per}(Y)$ the space of functions in $F_{loc}(\mathbb{R}^{N})$ that are $Y$-periodic, and by $F_{\#}(Y)$ those functions in $F_{per}(Y)$ with mean value zero. As special cases, $\mathcal{D}_{per}(Y)$ denotes the space $\mathcal{C}^{\infty}_{per}(Y)$ while $\mathcal{D}_{\#}(Y)$ stands for the space of those functions in $\mathcal{D}_{per}(Y)$ with mean value zero. $\mathcal{D}'_{per}(Y)$ stands for the topological dual of $\mathcal{D}_{per}(Y)$ which can be identified with the space of periodic distributions in $\mathcal{D}'(\mathbb{R}^{N})$. We say that an element $a \in L^{p}(\Omega\times\Lambda\times Y)$ is \textit{admissible} if the function $a_{\mathcal{T}} : (x, \omega, y) \rightarrow a(x, \mathcal{T}(x)\omega, y)$, $(x, \omega, y) \in \Omega\times\Lambda\times Y$, defines an element of $L^{p}(\Omega\times\Lambda\times Y)$. In this case, the trace function $(x,\omega)\mapsto a\left(x, \mathcal{T}\left(\frac{x}{\varepsilon}\right)\omega, \frac{x}{\varepsilon^{2}} \right)$ (denoted by $a^{\epsilon}$), from $\Omega\times\Lambda$ to $\mathbb{R}$ is well-defined as an element of $L^{p}(\Omega\times\Lambda)$ and satisfies the following convergence result :
	\begin{equation*}
		\int_{\Omega\times\Lambda} |a^{\varepsilon}| dxd\mu \rightarrow \iint_{\Omega\times\Lambda\times Y} |a| dxd\mu dy, \quad \textup{as}\; \varepsilon\to 0.
	\end{equation*}
	For example, any function in each of the following spaces is \textit{admissible} : 
	\begin{itemize}
		\item $\mathcal{C}^{\infty}_{0}(\Omega)\otimes\mathcal{C}^{\infty}(\Lambda)\otimes\mathcal{C}^{\infty}_{per}(Y) \subseteq L^{p}(\Omega)\otimes L^{p}(\Lambda) \otimes L^{p}_{per}(Y)$,
		\item $\mathcal{K}(\Omega; L^{p}(\Lambda;\mathcal{C}_{per}(Y)))$, $1\leq p \leq \infty$ the space of continuous functions of $\Omega$ into $L^{p}(\Lambda;\mathcal{C}_{per}(Y))$ with compact support contained in $\Omega$,
		\item $\mathcal{C}(\overline{\Omega}; L^{\infty}(\Lambda;\mathcal{C}_{per}(Y)))$, for any bounded domain $\Omega$ in $\mathbb{R}^{N}$. 
	\end{itemize}  
	Now we can define the concept of weak stochastic two-scale convergence.
	\begin{defn}
		Let $(u_{\varepsilon})_{\varepsilon>0}$ be a bounded sequence in $L^{p}(\Omega\times\Lambda)$, $1\leq p < \infty$. The sequence $(u_{\varepsilon})_{\varepsilon>0}$ is said to weakly stochastically two-scale converge in $L^{p}(\Omega\times\Lambda)$ to some $u_{0} \in L^{p}(\Omega\times\Lambda; L^{p}_{per}(Y))$ if as $\varepsilon\to 0$, we have 
		\begin{equation}
			\int_{\Omega\times\Lambda} u_{\varepsilon}(x,\omega)a\left(x, \mathcal{T}\left(\frac{x}{\varepsilon}\right)\omega, \frac{x}{\varepsilon^{2}} \right) dxd\mu \rightarrow \iint_{\Omega\times\Lambda\times Y} u_{0}(x,\omega,y) a(x,\omega,y) dx d\mu dy
		\end{equation}
		for every $a \in \mathcal{C}^{\infty}_{0}(\Omega)\otimes\mathcal{C}^{\infty}(\Lambda)\otimes\mathcal{C}^{\infty}_{per}(Y)$. \\
		We can express this by writing $u_{\varepsilon} \rightarrow u_{0}$ stoch. in $L^{p}(\Omega\times\Lambda)$-weak 2s. 
	\end{defn}
	
	We recall that $\mathcal{C}^{\infty}_{0}(\Omega)\otimes\mathcal{C}^{\infty}(\Lambda)\otimes\mathcal{C}^{\infty}_{per}(Y)$ is the space of functions of the form, 
	\begin{equation}
		a(x, \omega, y) = \sum_{finite} \varphi_{i}(x)\psi_{i}(\omega)g_{i}(y), \quad (x,\omega, y) \in \Omega\times\Lambda\times \mathbb{R}^{N} ,
	\end{equation}
	with $\varphi_{i} \in \mathcal{C}^{\infty}_{0}(\Omega)$, $\psi_{i} \in \mathcal{C}^{\infty}(\Lambda)$ and $g_{i} \in \mathcal{C}^{\infty}_{per}(Y)$. Such functions are dense in $\mathcal{C}^{\infty}_{0}(\Omega) \otimes L^{p'}(\Lambda)\otimes\mathcal{C}^{\infty}_{per}(Y)$ ($p'= \frac{p}{p-1}$) for $1< p < \infty$, since $\mathcal{C}^{\infty}(\Lambda)$ is dense in $L^{p'}(\Lambda)$ and hence $\mathcal{K}(\Omega; L^{p'}(\Lambda))\otimes\mathcal{C}^{\infty}_{per}(Y)$, where $\mathcal{K}(\Omega; L^{p'}(\Lambda))$ being the space of continuous functions of $\Omega$ into $L^{p'}(\Lambda)$ with compact support contained in $\Omega$, (see e.g., \cite[Chap III, Proposition 5]{barki2} for denseness result). As $\mathcal{K}(\Omega; L^{p'}(\Lambda))$ is dense in $L^{p'}(\Omega; L^{p'}(\Lambda)) = L^{p'}(\Omega\times\Lambda)$ and $L^{p'}(\Omega\times\Lambda)\otimes\mathcal{C}^{\infty}_{per}(Y)$ is dense in $L^{p'}(\Omega\times\Lambda ; \mathcal{C}_{per}(Y))$, the uniqueness of the stochastic two-scale limit is ensured. 
	\par Now, we give some properties about the stochastic two-scale convergence whose proof can be found in \cite{sango} or \cite{sango2}.
	\begin{prop}\cite{sango} \\
		Let $(u_{\varepsilon})_{\varepsilon>0}$ be a sequence in $L^{p}(\Omega\times\Lambda)$. If $u_{\varepsilon} \rightarrow u_{0}$ stoch. in $L^{p}(\Omega\times\Lambda)$-weak 2s, then $(u_{\varepsilon})_{\varepsilon>0}$ stochastically two-scale converges in the mean ( see \cite{bourgeat}) towards $v_{0}(x,\omega) = \int_{Y} u_{0}(x,\omega,y) dy$ and 
		\begin{equation}
			\int_{\Lambda} u_{\varepsilon}(\cdot, \omega)\psi(\omega)d\mu \rightarrow \iint_{\Lambda\times Y} u_{0}(\cdot, \omega, y)d\mu dy \; \textup{in} \, L^{1}(\Omega)-weak \; \forall \psi \in L^{p'}_{nv}(\Lambda).
		\end{equation}
	\end{prop}
	\begin{prop}\cite{sango}\label{rem3} \\
		Let $a \in \mathcal{K}(\Omega; \mathcal{C}_{per}(Y; \mathcal{C}^{\infty}(\Lambda)))$. Then, as $\varepsilon\to 0$, 
		\begin{equation}
			\int_{\Omega\times\Lambda} \left| a\left(x, \mathcal{T}\left(\frac{x}{\varepsilon}\right)\omega, \frac{x}{\varepsilon^{2}} \right)\right|^{p} dxd\mu \rightarrow \iint_{\Omega\times\Lambda\times Y} \left|a(x,\omega,y)\right|^{p} dx d\mu dy,
		\end{equation}
		for $1 \leq p < \infty$.
	\end{prop}
	\begin{prop}\cite{sango}\label{rem1} \\
		Any bounded sequence $(u_{\varepsilon})_{\varepsilon\in E}$  in $L^{p}(\Omega\times\Lambda)$ admits a subsequence which is weakly stochastically two-scale convergent in $L^{p}(\Omega\times\Lambda)$.
	\end{prop}
	\begin{prop}\cite{sango}\label{rem12} \\
		Let $1 < p < \infty$. Let $X$ be a norm closed convex subset of $W^{1,p}(\Omega)$, $\Omega$ being an open bounded subset of $\mathbb{R}^{N}$. Assume that $(u_{\varepsilon})_{\varepsilon\in E}$ is a sequence in $L^{p}(\Omega\times\Lambda)$ such that : 
		\begin{itemize}
			\item[i)] $u_{\varepsilon}(\cdot, \omega) \in X$ for all $\varepsilon \in E$ and for $\mu$-a.e. $\omega \in \Lambda$ ;
			\item[ii)] $(u_{\varepsilon})_{\varepsilon\in E}$ is bounded in $L^{p}(\Lambda; W^{1,p}(\Omega))$.
		\end{itemize}
		Then there exist $u_{0} \in W^{1,p}(\Omega; L_{nv}^{p}(\Lambda))$, $u_{1} \in L^{p}(\Omega; W^{1,p}_{\#}(\Lambda))$, $u_{2} \in L^{p}(\Omega\times \Lambda ; W^{1,p}_{\#}(Y))$ and a subsequence $E'$ from $E$ such that 
		\begin{itemize}
			\item[iii)] $u_{0}(\cdot, \omega) \in X$ for $\mu$-a.e. $\omega \in \Lambda$ and, as $E' \ni \epsilon \rightarrow 0$, 
			\item[iv)] $u_{\varepsilon} \rightarrow u_{0}$ stoch. in $L^{p}(\Omega\times\Lambda)$-weak 2s;
			\item[v)] $Du_{\varepsilon} \rightarrow Du_{0} + \overline{D}_{\omega}u_{1} + D_{y}u_{2}$ stoch. in $L^{p}(\Omega\times\Lambda)^{N}$-weak 2s.  
		\end{itemize}
	\end{prop}
	However, in order to deal with the convergence of a product of sequences, we need to define the concept of strong stochastic two-scale convergence.
	\begin{defn}
		A sequence $(u_{\varepsilon})_{\varepsilon>0}$ in $L^{p}(\Omega\times\Lambda)$, $1\leq p < \infty$ is said to strongly stochastically two-scale converge in $L^{p}(\Omega\times\Lambda)$ to some $u_{0} \in L^{p}(\Omega\times\Lambda; L^{p}_{per}(Y))$ if it is weakly stochastically two-scale convergent and further satisfies the following condition :  
		\begin{equation}
			\lim_{\varepsilon\to 0} \|u_{\varepsilon}\|_{L^{p}(\Omega\times\Lambda)} = \|u_{0}\|_{L^{p}(\Omega\times\Lambda\times Y)}.
		\end{equation}
		We can express this by writing $u_{\varepsilon} \rightarrow u_{0}$ stoch. in $L^{p}(\Omega\times\Lambda)$-strong 2s. 
	\end{defn}
	From the above definition, the uniqueness of the limit of such a sequence is ensured. The strong stochastic two-scale convergence is a generalization of the strong convergence as one can see in the following result.
	\begin{prop}\cite{sango2}\label{c2eq1}\\
		Let $(u_{\varepsilon})_{\varepsilon\in E} \subset L^{p}(\Omega\times\Lambda)$ $(1\leq p<\infty)$ be a strongly convergent sequence in $L^{p}(\Omega\times\Lambda)$ to some $u_{0}\in L^{p}(\Omega\times\Lambda)$. Then $(u_{\varepsilon})_{\varepsilon\in E}$ strongly stochastically two-scale converges in $L^{p}(\Omega\times\Lambda)$ towards $u_{0}$. 	
	\end{prop} 
	The next result gives the convergence criterion of a product of sequences and is really crucial in our homogenization process in section \ref{sect3}. 
	\begin{prop}\cite{sango2}\label{c2eq2}\\
		Let $1< p, q< \infty$ and $r\geq 1$ be such that $\frac{1}{p} + \frac{1}{q}=\frac{1}{r} \leq 1$. Assume that $(u_{\varepsilon})_{\varepsilon\in E} \subset L^{q}(\Omega\times\Lambda)$ is weakly stochastically
		two-scale convergent in $ L^{q}(\Omega\times\Lambda)$ to some $u_{0} \in  L^{q}(\Omega\times\Lambda; L^{q}_{per}(Y))$, and $(v_{\varepsilon})_{\varepsilon\in E} \subset L^{p}(\Omega\times\Lambda)$ is strongly stochastically
		two-scale convergent in $ L^{p}(\Omega\times\Lambda)$ to some $v_{0} \in  L^{p}(\Omega\times\Lambda; L^{p}_{per}(Y))$. Then the sequence $(u_{\varepsilon}v_{\varepsilon})_{\varepsilon\in E}$ is weakly stochastically two-scale convergence in $L^{r}(\Omega\times\Lambda)$ to $u_{0}v_{0}$. 
	\end{prop}
	\begin{cor}\cite{sango2}\label{c2eq3}  \\
		Let $(u_{\varepsilon})_{\varepsilon\in E} \subset L^{p}(\Omega\times\Lambda)$ and $(v_{\varepsilon})_{\varepsilon\in E} \subset L^{p'}(\Omega\times\Lambda)\cap L^{\infty}(\Omega\times\Lambda)$ $(1<p<\infty \, \textit{and}\, p'=p/(p-1))$ be two sequences such that : 
		\begin{itemize}
			\item[(i)] $u_{\varepsilon} \rightarrow u_{0}$ stoch. in $ L^{p}(\Omega\times\Lambda)$-weakly 2s;
			\item[(ii)] $v_{\varepsilon} \rightarrow v_{0}$ stoch. in $ L^{p'}(\Omega\times\Lambda)$-strong 2s;
			\item[(iii)] $(v_{\varepsilon})_{\varepsilon\in E}$ is bounded in $ L^{\infty}(\Omega\times\Lambda)$.
		\end{itemize}
		Then $u_{\varepsilon}v_{\varepsilon} \rightarrow u_{0}v_{0}$ stoch. in $L^{p}(\Omega\times\Lambda)$-weakly 2s.
	\end{cor}
	
	Looking at the second term on the left hand side of equation (\ref{c2tc11}), it is necessary for us to extend the concept of stochastic two-scale convergence of Sango and Woukeng \cite{sango} to periodic surfaces. This is the subject of the following subsection.
	
	\subsection{stochastic two-scale convergence on periodic surfaces}
	 Let $U$ be a bounded open set in $\mathbb{R}^{N}$, $Z$ the unit periodicity cell $[0,1]^{N}$ and $\mathcal{O}$ be an open subset of $Z$ with a smooth boundary $\Gamma$. We set $Z^{\star} = Z\backslash \overline{\mathcal{O}}$ and we identify $\mathcal{O}$, $Z^{\star}$ and $\Gamma$ with their images by their extension by $Z$-periodicity to the whole space $\mathbb{R}^{N}$. Then, for all $\varepsilon >0$, we define a perforated domain $U_{\varepsilon}$ by 
	\begin{equation}\label{doeq1}
		U_{\varepsilon} = \left\{x\in U \; / \; \frac{x}{\varepsilon} \in Z^{\star} \right\}.
	\end{equation}
	We further define a $(N-1)$-dimensional periodic surface $\Gamma_{\varepsilon}$ by 
	\begin{equation}\label{doeq2}
		\Gamma_{\varepsilon} = \left\{x\in U \; / \; \frac{x}{\varepsilon} \in \Gamma \right\},
	\end{equation}
	which is nothing else than the part $\partial U_{\varepsilon}$ lying inside $U$. We denote by $d\sigma(z), \, z\in Z$, and $d\sigma_{\varepsilon}(x), \, x\in U$, the surface measure on $\Gamma$, and $\Gamma_{\varepsilon}$ respectively. The spaces of $p$-integrable functions, with respect to these measures on $\Gamma$, and $\Gamma_{\varepsilon}$, are denoted by $L^{p}(\Gamma)$, and $L^{p}(\Gamma_{\varepsilon})$ respectively. The following theorem is a generalization to the periodic surfaces of \cite[Theorem 1]{sango}.
	
	\begin{thm}\label{pnptheo1}
		Let $(u_{\varepsilon})$ be a sequence in $L^{p}(\Gamma_{\varepsilon}\times\Lambda)$ such that 
		\begin{equation}\label{doeq3}
			\varepsilon \iint_{\Gamma_{\varepsilon}\times\Lambda} \left|u_{\varepsilon}(x,\omega)\right|^{p} d\sigma_{\varepsilon}(x) d\mu \leq C,
		\end{equation}
		where $C$ is a positive constant, independent of $\varepsilon$. Then, there exist a subsequence (still denoted by $\varepsilon$) and a stochastic two-scale limit $u_{0}(x,\omega,z) \in L^{p}\left(\Omega\times\Lambda; L^{p}_{per}(\Gamma)\right)$ such that $u_{\varepsilon}(x,\omega)$ stochastic two-scale converges to $u_{0}(x,\omega,z)$ in the sense that 
		\begin{equation}\label{doeq4}
			\begin{array}{l}
		\lim_{\varepsilon\to 0} \varepsilon \iint_{\Gamma_{\varepsilon}\times\Lambda} u_{\varepsilon}(x,\omega) \varphi\left(x, \mathcal{T}\left(\frac{x}{\varepsilon}\right), \frac{x}{\varepsilon^{2}}\right) d\sigma_{\varepsilon}(x) d\mu \\
		 = \int_{U}\iint_{\Gamma\times\Lambda} u_{0}(x,\omega,z) \varphi(x,\omega,z) dx d\sigma(z) d\mu,	
	\end{array}
		\end{equation}
		for any test function $\varphi(x,\omega,z) \in     \mathcal{C}^{\infty}_{0}(U)\times \mathcal{C}^{\infty}(\Lambda)\times\mathcal{C}_{per}(Z)$.
	\end{thm}
	The proof of this theorem is very similar to the usual case (see \cite[Theorem 1]{sango}). It relies on the following lemma and proposition. 
	\begin{lem}\label{dolem1}
		Let $a \in \mathcal{K}\left(U; \mathcal{C}^{\infty}\left(\Lambda; \mathcal{C}_{per}(Z)\right)\right)$. Then,  
		\begin{equation}\label{doeq5}
			\lim_{\varepsilon\to 0} \varepsilon \iint_{\Gamma_{\varepsilon}\times\Lambda} \left| a\left(x, \mathcal{T}\left(\frac{x}{\varepsilon}\right), \frac{x}{\varepsilon^{2}}\right)\right|^{p} d\sigma_{\varepsilon}(x) d\mu = \int_{U}\iint_{\Gamma\times\Lambda} \left| a(x,\omega,z)\right|^{p} dx d\sigma(z) d\mu.	
		\end{equation}
	\end{lem} 
	\begin{proof}
	The proof is done in two steps.\\
	\textbf{step 1} Firstly, we will prove that 
	\begin{equation}\label{doeq6}
		\lim_{\varepsilon\to 0} \varepsilon \iint_{\Gamma_{\varepsilon}\times\Lambda}  a\left(x, \mathcal{T}\left(\frac{x}{\varepsilon}\right), \frac{x}{\varepsilon^{2}}\right) d\sigma_{\varepsilon}(x) d\mu = \int_{U}\iint_{\Gamma\times\Lambda}  a(x,\omega,z) dx d\sigma(z) d\mu,	
	\end{equation}
	for all $a \in \mathcal{K}\left(U; \mathcal{C}^{\infty}\left(\Lambda; \mathcal{C}_{per}(Z)\right)\right)$. \\
	Since $ \mathcal{C}^{\infty}_{0}(U)\times \mathcal{C}^{\infty}(\Lambda)\times\mathcal{C}_{per}(Z)$ is dense in $\mathcal{K}\left(U; \mathcal{C}^{\infty}\left(\Lambda; \mathcal{C}_{per}(Z)\right)\right)$, we first check (\ref{doeq6}) for $ a \in \mathcal{C}^{\infty}_{0}(U)\times \mathcal{C}^{\infty}(\Lambda)\times\mathcal{C}_{per}(Z)$. For that, it is sufficient to do it for $a$ under the form $a(x,\omega,z) = f(x)g(\omega)h(y)$ with $f \in \mathcal{C}^{\infty}_{0}(U)$, $g\in \mathcal{C}^{\infty}(\Lambda)$ and $h\in \mathcal{C}_{per}(Z)$. But for such an $a$, we have 
	\begin{equation}\label{doeq7}
		\begin{array}{l}
		 \varepsilon \iint_{\Gamma_{\varepsilon}\times\Lambda}  a\left(x, \mathcal{T}\left(\frac{x}{\varepsilon}\right), \frac{x}{\varepsilon^{2}}\right) d\sigma_{\varepsilon}(x) d\mu \\
		  =  \varepsilon \int_{\Gamma_{\varepsilon}} \left(\int_{\Lambda} g\left( \mathcal{T}\left(\frac{x}{\varepsilon}\right)\omega \right) d\mu\right) f(x) h\left( \frac{x}{\varepsilon^{2}} \right) d\sigma_{\varepsilon}(x) 	\\
		   =  \varepsilon \int_{\Gamma_{\varepsilon}} \left(\int_{\Lambda} g(\omega) d\mu\right) f(x) h\left( \frac{x}{\varepsilon^{2}} \right) d\sigma_{\varepsilon}(x)  \\
		   =  \left(\int_{\Lambda} g(\omega) d\mu\right)\times \varepsilon \int_{\Gamma_{\varepsilon}} f(x) h\left( \frac{x}{\varepsilon^{2}} \right) d\sigma_{\varepsilon}(x),
		\end{array}
	\end{equation}
	where the second equality above is due to the Fubini's theorem and to the fact that the measure $\mu$ is invariant under the maps $\mathcal{T}(z)$. But, as $\varepsilon\to 0$, we thanks to \cite[Lemma 2.4]{allair3} and we have the following convergence result : 
	\begin{equation*}
	 \varepsilon \int_{\Gamma_{\varepsilon}} f(x) h\left( \frac{x}{\varepsilon^{2}} \right) d\sigma_{\varepsilon}(x) \rightarrow \iint_{U\times\Gamma} f(x) h(z) dx d\sigma(z) \quad \textup{as} \quad \varepsilon \to 0.	
	\end{equation*}
	Hence, passing to the limit in (\ref{doeq7}),  we get 
	\begin{equation*}
		\begin{array}{l}
		\lim_{\varepsilon\to 0}	\varepsilon \iint_{\Gamma_{\varepsilon}\times\Lambda}  a\left(x, \mathcal{T}\left(\frac{x}{\varepsilon}\right), \frac{x}{\varepsilon^{2}}\right) d\sigma_{\varepsilon}(x) d\mu  \\
		 =   
		 \left(\int_{\Lambda} g(\omega) d\mu\right)\times \iint_{U\times\Gamma} f(x) h(z) dx d\sigma(z) 	\\
			 =  \int_{U}\iint_{\Gamma\times\Lambda}  a(x,\omega,z) dx d\sigma(z) d\mu. 
		\end{array}
	\end{equation*}
	Now arguing as in the proof of \cite[Proposition 3]{sango}, we conclude with the standard density argument that (\ref{doeq6}) is holds for all $a \in \mathcal{K}\left(U; \mathcal{C}^{\infty}\left(\Lambda; \mathcal{C}_{per}(Z)\right)\right)$. 
	\textbf{Step 2.} Let now $1\leq p < \infty$. Since $\mathcal{K}\left(U; \mathcal{C}^{\infty}\left(\Lambda; \mathcal{C}_{per}(Z)\right)\right)$ is a Banach algebra then we have that, $|a|^{p} \in \mathcal{K}\left(U; \mathcal{C}^{\infty}\left(\Lambda; \mathcal{C}_{per}(Z)\right)\right)$ whenever $a \in \mathcal{K}\left(U; \mathcal{C}^{\infty}\left(\Lambda; \mathcal{C}_{per}(Z)\right)\right)$, so that, as $\varepsilon\to 0$, 
	  \begin{equation}
	  	 \varepsilon \iint_{\Gamma_{\varepsilon}\times\Lambda} \left| a\left(x, \mathcal{T}\left(\frac{x}{\varepsilon}\right), \frac{x}{\varepsilon^{2}}\right)\right|^{p} d\sigma_{\varepsilon}(x) d\mu = \int_{U}\iint_{\Gamma\times\Lambda} \left| a(x,\omega,z)\right|^{p} dx d\sigma(z) d\mu.	
	  \end{equation}
	\end{proof}
	
	\begin{prop}\cite[Proposition 3.2]{gabri}\label{doprop1} \\
		Let $F$ be a subspace (not necessarily closed) of a reflexive Banach space $G$ and let $f_{n} : F \rightarrow \mathbb{C}$ be a sequence of linear functionals (not necessarily continuous).  Assume there exists a constant $C > 0$ such that 
		\begin{equation*}
			\limsup_{n\to \infty} |f_{n}(x)| \leq C \, \|x\| \quad \textup{for \,\, all} \,\, x \in F,
		\end{equation*}
		where $\|\cdot\|$ denotes the norm in $G$. Then there exists a subsequence $(f_{n_{k}})_{k}$ of $(f_{n})$ and a functional $f \in G'$ such that 
		\begin{equation*}
			\lim_{k} f_{n_{k}}(x) = f(x) \quad \textup{for \,\, all} \,\, x \in F,
		\end{equation*} 
	\end{prop} 
	
	\begin{proof}(of Theorem \ref{pnptheo1})  \\
	Let $G = L^{p'}\left(U\times\Lambda; L^{p'}_{per}(\Gamma)\right)$, and $F = \mathcal{C}^{\infty}_{0}(U)\times \mathcal{C}^{\infty}(\Lambda)\times\mathcal{C}_{per}(Z)$. Let us define the mapping $\Psi_{\varepsilon}$ by 
	\begin{equation*}
		\Psi_{\varepsilon}(a) = \varepsilon \iint_{\Gamma_{\varepsilon}\times\Lambda} u_{\varepsilon} a^{\varepsilon}  d\sigma_{\varepsilon}(x) d\mu \quad a \in \mathcal{C}^{\infty}_{0}(U)\times \mathcal{C}^{\infty}(\Lambda)\times\mathcal{C}_{per}(Z) 
	\end{equation*}
	where $a^{\varepsilon}(x,\omega) = a\left(x, \mathcal{T}\left(\frac{x}{\varepsilon}\right), \frac{x}{\varepsilon^{2}}\right)$ for $(x,\omega) \in \Gamma_{\varepsilon}\times\Lambda$. Then by Holder's inequality and using (\ref{doeq3}), we have
	\begin{equation*}
		 \left|\Psi_{\varepsilon}(a) \right| \leq C^{1/p} \left[\varepsilon \iint_{\Gamma_{\varepsilon}\times\Lambda} |a^{\varepsilon}(x,\omega)|^{p'} d\sigma_{\varepsilon}(x) d\mu \right]^{1/p'}\quad \textup{for\, all}\; a \in F,
	\end{equation*}
	and thus, by Lemma \ref{dolem1}, one has 
	\begin{equation*}
		\begin{array}{rcl}	
	\limsup_{\varepsilon\to 0}	\left|\Psi_{\varepsilon}(a) \right| & \leq & C^{1/p} \left[\int_{U}\iint_{\Gamma\times\Lambda} \left| a(x,\omega,z)\right|^{p'} dx d\sigma(z) d\mu  \right]^{1/p'}  \\ 
	& \leq &  C^{1/p} \left\|a \right\|_{L^{p'}(U\times\Lambda\times\Gamma)}. 
\end{array}
	\end{equation*}
	Hence, we deduce from the Proposition \ref{doprop1} the existence of a subsequence  $E'$ of $E$ and of a unique $u_{0} \in G'= L^{p}\left(U\times\Lambda; L^{p}_{per}(\Gamma)\right)$ such that 
	\begin{equation*}
\lim_{\varepsilon\to 0}	\varepsilon \iint_{\Gamma_{\varepsilon}\times\Lambda} u_{\varepsilon} a^{\varepsilon}  d\sigma_{\varepsilon}(x) d\mu = \int_{U}\iint_{\Gamma\times\Lambda} u_{0}(x,\omega,z) a(x,\omega,z) dx d\sigma(z) d\mu,	
\end{equation*}
for any  $a \in     \mathcal{C}^{\infty}_{0}(U)\times \mathcal{C}^{\infty}(\Lambda)\times\mathcal{C}_{per}(Z)$.	
	\end{proof}
	
	We have recalled all the tools to study the stochastic-periodic homogenization of Poisson-Nernst-Plank equations (\ref{c2tc10})-(\ref{c2tc12}) in porous media.
	

	\section{Global homogenized problem  s-PNP}\label{sect3}
	
	The aim of this section is to prove the following Theorem.	
	\begin{thm}\label{nptheo1}	
	Let $(\varpi_{\pm,\varepsilon}, \varUpsilon_{\varepsilon})$ be the unique solution of weak problem (\ref{c2tc10})-(\ref{c2tc12})  such that $(\varpi_{\pm,\varepsilon}) = \left(\varpi_{+,\varepsilon}, \varpi_{-,\varepsilon}\right) \in L^{\infty}\left([0,T];L^{2}(\Lambda; H^{1}(\Omega_{\varepsilon}^{f}))\right)^{2} \cap L^{2}\left(\Lambda;H^{1}(Q_{\varepsilon}^{f})\right)^{2}$ and $\varUpsilon_{\varepsilon} \in \mathcal{C}^{0}\left([0,T]; L^{2}(\Lambda; H^{1}(\Omega))\right)$. Then for all $t>0$, we have as $\varepsilon\to 0$ :
	
	\begin{equation}\label{c4tc1}
		\varpi_{\pm,\varepsilon} \rightarrow \varpi_{\pm,0} \;\, stoch.\, in \; L^{2}(Q\times\Lambda)-weak\, 2s
	\end{equation}
	
	\begin{equation}\label{c4tc2}
		\nabla\varpi_{\pm,\varepsilon} \rightarrow \nabla_{x}\varpi_{\pm,0} + \overline{D}_{\omega}\varpi_{\pm,1} + \nabla_{y}\varpi_{\pm,2} \;\, stoch.\, in \; L^{2}(Q\times\Lambda)^{N}-weak\, 2s
	\end{equation}
	
	\begin{equation}\label{c4tc3}
		\varUpsilon_{\varepsilon}(t) \rightarrow \varUpsilon_{0}(t) \;\, stoch.\, in \; L^{2}(\Omega\times\Lambda)-weak\, 2s
	\end{equation}
	
	\begin{equation}\label{c4tc4}
		\nabla\varUpsilon_{\varepsilon}(t) \rightarrow \nabla_{x}\varUpsilon_{0}(t) + \overline{D}_{\omega}\varUpsilon_{1}(t) + \nabla_{y}\varUpsilon_{2}(t) \;\, stoch.\, in \; L^{2}(\Omega\times\Lambda)^{N}-weak\, 2s
	\end{equation}
	
	\begin{equation}\label{c4tc}
	\varpi_{\pm,0}(0,\cdot,\cdot) = \varpi^{0}_{\pm,0} \quad \textup{in} \; \Omega\times\Lambda.	
	\end{equation}
where : 
\begin{itemize}
	\item $(\varpi_{\pm,0},\varUpsilon_{0}(t)) \in H^{1}(Q; L^{2}_{nv}(\Lambda))\times H^{1}(\Omega; L^{2}_{nv}(\Lambda))$,
	\item $(\varpi_{\pm,1},\varUpsilon_{1}(t)) \in L^{2}(Q ; H^{1}_{\#}(\Lambda))\times L^{2}(\Omega ; H^{1}_{\#}(\Lambda))$ and 
	\item $(\varpi_{\pm,2},\varUpsilon_{2}(t)) \in L^{2}(Q\times\Lambda ; H^{1}_{\#}(Y))\times L^{2}(\Omega\times\Lambda ; H^{1}_{\#}(Y))$
\end{itemize}
    are the unique solutions to the variational system\footnote{It is referred to as the \textit{global homogenized problem} for (\ref{c1tc1})-(\ref{c1tc3})} (\ref{c4tc5})-(\ref{c4tc6}):
	\begin{equation}\label{c4tc5}
	\begin{array}{l}
		-\int_{Q\times\Lambda\times Y} \chi_{Y_{f}}(y) \varpi_{\pm,0} \dfrac{\partial \phi_{0} }{\partial t}  dxdtd\mu dy \\
		
		\\
		\quad + D_{\pm} \int_{Q\times\Lambda\times Y} \chi_{Y_{f}}(y) \left(\nabla_{x}\varpi_{\pm,0} + \overline{D}_{\omega}\varpi_{\pm,1} + \nabla_{y}\varpi_{\pm,2}\right)  \\
		 \qquad \qquad \quad \qquad \quad \times \left(\nabla_{x}\phi_{0} + \overline{D}_{\omega}\phi_{1} + \nabla_{y}\phi_{2}\right) dxdtd\mu dy \\
		
		\\
		\quad + D_{\pm}\dfrac{F}{R\varTheta} \int_{Q\times\Lambda\times Y} \chi_{Y_{f}}(y) z_{\pm}\varpi_{\pm,0}  \left(\nabla_{x}\varUpsilon_{0} + \overline{D}_{\omega}\varUpsilon_{1} + \nabla_{y}\varUpsilon_{2}\right)   \\
		\qquad \qquad \quad \qquad \quad \qquad \times\left(\nabla_{x}\phi_{0} + \overline{D}_{\omega}\phi_{1} + \nabla_{y}\phi_{2}\right) dxdtd\mu dy = 0,
	\end{array}
\end{equation}
\vspace{1cm}
\begin{equation}\label{c4tc6}
	\begin{array}{l}
		\int_{\Omega\times\Lambda}\int_{Y} \vartheta\left(\omega,y\right) \left(\nabla_{x}\varUpsilon_{0} + \overline{D}_{\omega}\varUpsilon_{1} + \nabla_{y}\varUpsilon_{2}\right)\left(\nabla_{x}\psi_{0} + \overline{D}_{\omega}\psi_{1} + \nabla_{y}\psi_{2}\right) dxd\mu dy  \\
		
		\\
		\quad	+  \int_{\Omega}\int_{\Gamma^{sf}\times\Lambda} \eta\left(\omega,y\right) \gamma\left(\varUpsilon_{0}(t)\right) \psi_{0} dxdS(y)d\mu  \\
		
		\\
		\quad	=
		\int_{\Omega\times\Lambda}\int_{Y} \chi_{Y_{f}}(y)	F \left(z_{+}\varpi_{+,0}(t) - z_{-}\varpi_{-,0}(t)\right)\psi_{0} dxd\mu dy, 
	\end{array}
\end{equation}
\vspace{0.2cm}
where $\phi_{0} \in \mathcal{D}(]0,T[)\times\mathcal{C}_{0}^{\infty}(Q)\times L_{nv}^{2}(\Lambda)$, $\phi_{1} \in \mathcal{D}(]0,T[)\times\mathcal{C}_{0}^{\infty}(Q)\times \mathcal{C}^{\infty}(\Lambda)$, $\phi_{2} \in \mathcal{D}(]0,T[)\times\mathcal{C}_{0}^{\infty}(Q)\times \mathcal{C}^{\infty}(\Lambda)\times \mathcal{D}_{\#}(Y)$
 and $\psi_{0} \in \mathcal{C}_{0}^{\infty}(\Omega)\times L_{nv}^{2}(\Lambda)$, $\psi_{1} \in \mathcal{C}_{0}^{\infty}(\Omega)\times \mathcal{C}^{\infty}(\Lambda)$, $\psi_{2} \in \mathcal{C}_{0}^{\infty}(\Omega)\times \mathcal{C}^{\infty}(\Lambda)\times \mathcal{D}_{\#}(Y)$.  
	\end{thm}

\begin{proof}
	\textbf{i) \textit{A priori} uniform estimates} \\
	 Arguing as in \cite[Page 78]{gagneu1}, let us introduce the extension operators $\mathcal{P}_{\varepsilon}$ in $Q$ outside of $Q_{\varepsilon}^{f}$, equi-bounded of $L^{2}(\Lambda;H^{1}(Q_{\varepsilon}^{f}))$ in $L^{2}(\Lambda;H^{1}(Q))$. We then extend the functions $\varpi_{\pm,\varepsilon}$ into $\widetilde{\varpi_{\pm,\varepsilon}} = \mathcal{P}_{\varepsilon}(\varpi_{\pm,\varepsilon})$ and we have \textit{a priori} estimates, for $\mu$-a.e. $\omega\in\Lambda$ : 
	\begin{equation}\label{c4tc7}
	\|\widetilde{\varpi_{\pm,\varepsilon}}(\cdot,\cdot,\omega)\|_{H^{1}(Q)\cap L^{\infty}([0,T];H^{1}(\Omega))} \leq C,	
	\end{equation}
	\begin{equation}\label{c4tc8}
		\begin{array}{l}
	\|\gamma\left(\varUpsilon_{\varepsilon}(\cdot,\cdot,\omega)\right)\|_{\mathcal{C}^{0}([0,T];H^{1}(\Omega))} \\
	\qquad + \max_{t\in[0,T]} \varepsilon\int_{\Gamma^{sf}_{\varepsilon}}	 \left|\eta\left(\mathcal{T}\left(\frac{\cdot}{\varepsilon}\right)\omega,\frac{\cdot}{\varepsilon^{2}}\right) \gamma\left(\varUpsilon_{\varepsilon}(t,\cdot,\omega)\right)\right|^{2} dS_{\varepsilon} \leq C,
\end{array}
	\end{equation}
	for some constant $C$ independent of $\varepsilon$. \\
	\textbf{ii) stochastic two-scale convergence of concentrations}  \\
	 Using uniform bounded (\ref{c4tc7}) and Theorem \ref{rem12}, from the sequence $\widetilde{\varpi_{\pm,\varepsilon}}$ we can extract a subsequence, still noted $\widetilde{\varpi_{\pm,\varepsilon}}$ such that\footnote{Note that in Theorem \ref{nptheo1} for the ease of presentation we considered the solution $\varpi_{\pm,\varepsilon}$ instead of its extension $\widetilde{\varpi_{\pm,\varepsilon}}$ }, as $\varepsilon\to 0$ : 
	 \begin{equation*}
	 	\widetilde{\varpi_{\pm,\varepsilon}} \rightarrow \varpi_{\pm,0} \;\, stoch.\, in \; L^{2}(Q\times\Lambda)-weak\, 2s,
	 \end{equation*}
	 \begin{equation*}
	 	\nabla\widetilde{\varpi_{\pm,\varepsilon}} \rightarrow \nabla_{x}\varpi_{\pm,0} + \overline{D}_{\omega}\varpi_{\pm,1} + \nabla_{y}\varpi_{\pm,2} \;\, stoch.\, in \; L^{2}(Q\times\Lambda)^{N}-weak\, 2s,
	 \end{equation*}
	 with $\varpi_{\pm,0} \in H^{1}(Q; L^{2}_{nv}(\Lambda))$, $\varpi_{\pm,1} \in L^{2}(Q ; H^{1}_{\#}(\Lambda))$, $\varpi_{\pm,2} \in L^{2}(Q\times\Lambda ; H^{1}_{\#}(Y))$. Moreover $\widetilde{\varpi_{\pm,\varepsilon}}$ converges to $\varpi_{\pm,0}$ as follows : 
	 \begin{itemize}
	 	\item weakly in $L^{2}(\Lambda; H^{1}(Q))^{2}$,
	 	\item weakly-$\ast$ in $L^{\infty}(0,T ; L^{2}(\Lambda; H^{1}(Q)))^{2}$,
	 	\item  strongly in $L^{2}(\Lambda; L^{2}(Q))^{2}$ and $\mathcal{C}^{0}([0,T]; L^{2}(\Lambda; L^{2}(Q)))^{2}$,
	 	\item $\mu\times\mathcal{L}^{N+1}$-a.e. in $\Lambda\times Q$.
	 \end{itemize}
	 \textbf{iii) stochastic two-scale convergence of potential}  \\
	 Using uniform bounded (\ref{c4tc8}) for $t>0$ fixed, the expression $\varUpsilon_{\varepsilon}(t)$ has a sense and defines an element of $L^{2}(\Lambda; H^{1}(\Omega))$ uniformly bounded with respect to $t$. Thanks to Theorem \ref{rem12}, we can extract a subsequence $\varUpsilon_{\varepsilon(t)}$ \textit{a priori} dependent of $t$, such that as $\varepsilon(t)\to 0$ : 
	 \begin{equation*}
	 	\varUpsilon_{\varepsilon(t)} \rightarrow \varUpsilon_{0}(t) \;\, stoch.\, in \; L^{2}(\Omega\times\Lambda)-weak\, 2s,
	 \end{equation*}
	 \begin{equation*}
	 	\nabla\varUpsilon_{\varepsilon(t)} \rightarrow \nabla_{x}\varUpsilon_{0}(t) + \overline{D}_{\omega}\varUpsilon_{1}(t) + \nabla_{y}\varUpsilon_{2}(t) \;\, stoch.\, in \; L^{2}(\Omega\times\Lambda)^{N}-weak\, 2s,
	 \end{equation*}
	 with  $\varUpsilon_{0}(t) \in  H^{1}(\Omega; L^{2}_{nv}(\Lambda))$, $\varUpsilon_{1}(t) \in  L^{2}(\Omega ; H^{1}_{\#}(\Lambda))$ and $\varUpsilon_{2}(t) \in  L^{2}(\Omega\times\Lambda ; H^{1}_{\#}(Y))$.  Moreover $\varUpsilon_{\varepsilon(t)}$ converges to $\varUpsilon_{0}(t)$ as follows : 
	 \begin{itemize}
	 	\item weakly in $L^{2}(\Lambda; H^{1}(\Omega))$,
	 	\item  strongly in $L^{2}(\Lambda; L^{2}(\Omega))$,
	 \end{itemize}
	 and since $\gamma$ is a Lipschitz function, $\gamma(\varUpsilon_{\varepsilon(t)})$ converges to $\gamma(\varUpsilon_{0}(t))$ weakly in $L^{2}(\Lambda; H^{1}(\Omega))$ and strongly in $L^{2}(\Lambda; L^{2}(\Omega))$. To end, using Theorem \ref{pnptheo1} and uniform bounded (\ref{c4tc8}), we have that for $t>0$ fixed, $\eta\left(\mathcal{T}\left(\frac{x}{\varepsilon(t)}\right)\omega,\frac{x}{\varepsilon^{2}(t)}\right) \gamma\left(\varUpsilon_{\varepsilon(t)}(t)\right)_{|\Gamma_{\varepsilon(t)}^{sf}\times\Lambda}$ stochastic two-scale converges to $\eta\left(\omega,y\right) \gamma\left(\varUpsilon_{0}(t,x,\omega)\right)$ as $\varepsilon(t)\to 0$.\\
	 \textbf{iv) Established (\ref{c4tc5})-(\ref{c4tc6})}  \\
	 Firsly let us rewrite the system (\ref{c2tc10})-(\ref{c2tc11}) with integrals in $\Lambda\times Q$ and $\Lambda\times \Omega$, for any test functions $\phi\in \mathcal{D}(]0,T[; \mathcal{C}^{\infty}(\Lambda; V))$ and $\psi\in  \mathcal{C}^{\infty}(\Lambda; V)$ :
	 
	 \begin{equation}\label{c4tc9}
	 	\begin{array}{l}
	 		\int_{Q\times\Lambda} \chi_{\Omega_{\varepsilon}^{f}} \widetilde{\varpi_{\pm,\varepsilon}} \dfrac{\partial \phi}{\partial t}  dxdtd\mu + D_{\pm} \int_{Q\times\Lambda} \chi_{\Omega_{\varepsilon}} \nabla\widetilde{\varpi_{\pm,\varepsilon}}.\nabla\phi dxdtd\mu \\
	 		
	 		\\
	 		\quad + D_{\pm}\dfrac{F}{R\varTheta} \int_{Q\times\Lambda} \chi_{\Omega_{\varepsilon}^{f}} z_{\pm}\widetilde{\varpi_{\pm,\varepsilon}}\nabla\varUpsilon_{\varepsilon}.\nabla\phi dxdtd\mu = 0,
	 	\end{array}
	 \end{equation}
	 \vspace{0.5cm}
	 \begin{equation}\label{c4tc10}
	 	\begin{array}{l}
	 		\int_{\Omega\times\Lambda} \vartheta\left(\mathcal{T}\left(\frac{x}{\varepsilon}\right)\omega,\frac{x}{\varepsilon^{2}}\right)\nabla\varUpsilon_{\varepsilon}(t).\nabla\psi dxd\mu  \\
	 		
	 		\\
	 		\quad	+ \varepsilon \iint_{\Gamma^{sf}_{\varepsilon}\times \Lambda} \eta\left(\mathcal{T}\left(\frac{x}{\varepsilon}\right)\omega,\frac{x}{\varepsilon^{2}}\right) \gamma\left(\varUpsilon_{\varepsilon}(t)\right) \psi dS_{\varepsilon}d\mu  \\ 
	 		
	 		\\
	 		\quad	=
	 		F \int_{\Omega\times\Lambda} \chi_{\Omega_{\varepsilon}^{f}} \left(z_{+}\widetilde{\varpi_{+,\varepsilon}}(t) - z_{-}\widetilde{\varpi_{-,\varepsilon}}(t)\right)\psi dxd\mu, 
	 	\end{array}
	 \end{equation}
	 where $t>0$ is fixed in (\ref{c4tc10}). Now passing to the (stochastic) two-scale limit with the test functions of the form $\phi(t,x,\omega) = \phi_{0}(t,x,\omega) + \varepsilon\phi_{1}\left(t,x,\mathcal{T}\left(\frac{x}{\varepsilon}\right)\omega\right) + \varepsilon^{2}\phi_{2}\left(t,x,\mathcal{T}\left(\frac{x}{\varepsilon}\right)\omega, \frac{x}{\varepsilon^{2}}\right)$ in (\ref{c4tc9}) and $\psi(x,\omega) = \psi_{0}(x,\omega) + \varepsilon\psi_{1}\left(x,\mathcal{T}\left(\frac{x}{\varepsilon}\right)\omega\right) + \varepsilon^{2}\psi_{2}\left(x,\mathcal{T}\left(\frac{x}{\varepsilon}\right)\omega, \frac{x}{\varepsilon^{2}}\right)$ in (\ref{c4tc10}), we obtain the relations (\ref{c4tc5}) and (\ref{c4tc6}). Futhermore (\ref{c4tc}) follows from the stochastically two-scale convergence in the Cauchy condition (\ref{c2tc12}). \\
	 \textbf{v) Uniqueness of solutions} 
	 To end, by the same arguments as in \cite[Page 81]{gagneu1}, the uniqueness of solutions to variational problems (\ref{c4tc5})-(\ref{c4tc6}) derive from the application of monotone methods (see, e.g. \cite[Chapitre 2]{jlion2}) and from Gronwall lemma. Thus we deduce \textit{a posteriori} the convergence of the sequences $\widetilde{\varpi_{\pm,\varepsilon}}$ and $\varUpsilon_{\varepsilon}(t)$ without needing to extract subsequences.
\end{proof}	
\begin{rem}
	We deduce that the homogenized problem obtained from the Poisson weak equation (\ref{c4tc6}) reduces to :
	\begin{equation*}
		\begin{array}{r}
		-\textup{div}_{y} \left[\vartheta\left(\omega,y\right) \left(\nabla_{x}\varUpsilon_{0}(t,x,\omega) + \overline{D}_{\omega}\varUpsilon_{1}(t,x,\omega) + \nabla_{y}\varUpsilon_{2}(t,x,\omega,y)\right)\right]=0 \\
		 \textup{in}\; \Omega\times\Lambda\times Y,
		\end{array}
	\end{equation*}
	\vspace{0.2cm}
	\begin{equation*}
		\begin{array}{r}
		-\textup{div}_{\omega} \left[\int_{Y}\vartheta\left(\omega,y\right) \left(\nabla_{x}\varUpsilon_{0}(t,x,\omega) + \overline{D}_{\omega}\varUpsilon_{1}(t,x,\omega) + \nabla_{y}\varUpsilon_{2}(t,x,\omega,y)\right)dy \right]=0 \\ \textup{in}\; \Omega\times\Lambda,
	\end{array}
	\end{equation*}
	\vspace{0.3cm}
	\begin{equation*}
		\begin{array}{l}
		-\textup{div}_{x} \left[\int_{Y}\vartheta\left(\omega,y\right) \left(\nabla_{x}\varUpsilon_{0}(t,x,\omega) + \overline{D}_{\omega}\varUpsilon_{1}(t,x,\omega) + \nabla_{y}\varUpsilon_{2}(t,x,\omega,y)\right)dy \right]  \\ 
		
		\\
	  \quad	+ \gamma\left(\varUpsilon_{0}(t,x,\omega)\right)\int_{\Gamma^{sf}} \eta\left(\omega,y\right) dS(y) \\
		
		\\
		\quad  = 
			F \left(z_{+}\varpi_{+,0}(t,x,\omega) - z_{-}\varpi_{-,0}(t,x,\omega)\right) \int_{Y} \chi_{Y_{f}}(y) dy \;\; \quad \textup{in}\; \Omega\times\Lambda,
		\end{array}
	\end{equation*}
	with $t>0$ fixed.
\end{rem}
	
\begin{rem}
	As in \cite[Page 78]{gagneu1}, we have the following properties of physical equilibrium of concentrations and electrical potential, for $t>0$ : 
		\begin{equation}\label{c4tc11}
		(\varpi_{\pm,0}) \in L^{2}(\Lambda; L^{\infty}(Q_{\varepsilon}^{f}))^{2}, \quad \varpi_{\pm,0}(t,\cdot,\cdot) \geq 0 \quad \textup{in} \;\, \Omega\times\Lambda,
	\end{equation}
	
	\begin{equation}\label{c4tc12}
		\iint_{\Omega\times\Lambda} \varpi_{\pm,0}(t,x,\omega) dxd\mu = \iint_{\Omega\times\Lambda} \varpi^{0}_{\pm,0}(x,\omega) dxd\mu,
	\end{equation}
	as consequence of (\ref{c4tc5}) with the test function $\phi = 1$,
	
	\begin{equation}\label{c4tc13}
		\begin{array}{l}
		F \left(\int_{Y} \chi_{Y_{f}}(y)dy\right) \int_{\Omega\times\Lambda} \left(z_{+}\varpi_{+,\varepsilon}(t,x,\omega) - z_{-}\varpi_{-,\varepsilon}(t,x,\omega)\right) dxd\mu \\
		
		\\
		=  \int_{\Omega}\iint_{\Gamma^{sf}\times\Lambda} \eta\left(\omega,y\right) \gamma\left(\varUpsilon_{0}(t)\right)  dxdS(y)d\mu  
\end{array}
	\end{equation}
	as consequence of (\ref{c4tc6}) with the test function $\psi = 1$. Therefore, (\ref{c4tc12}) implies that the quantity
	\begin{equation*}
		\varPi(t) := \int_{\Omega}\iint_{\Gamma^{sf}\times\Lambda} \eta\left(\omega,y\right) \gamma\left(\varUpsilon_{0}(t)\right)  dxdS(y)d\mu	
	\end{equation*}
	is constant with respect to time $t$ and satisfies to the global electric equilibrium : 
	
	\begin{equation}\label{c4tc14}
		\varPi(t) =  F \left(\int_{Y}\chi_{Y_{f}}(y)dy\right) \left[- \int_{\Omega\times\Lambda} z_{+}\varpi^{0}_{+,0}(x,\omega)dxd\mu  +  \int_{\Omega\times\Lambda} z_{-}\varpi^{0}_{-,0}(x,\omega) dxd\mu \right].
	\end{equation}
\end{rem}	
		
	\begin{rem}
		Indeed, when we consider the particular dynamical system $\mathcal{T}(y)$ on $\Lambda = \mathbb{T}^{N} \equiv \mathbb{R}^{N}/\mathbb{Z}^{N}$ (the $N$-dimensional torus) defined by $\mathcal{T}(y)\omega = y+\omega\;\textup{mod}\;\mathbb{Z}^{N}$, then one can view $\Lambda$ as the unit cube in $\mathbb{R}^{N}$ with all the pairs of antipodal faces being identified. The Lebesgue measure on $\mathbb{R}^{N}$ induces the Haar measure on $\mathbb{T}^{N}$ which is invariant with respect to the action of $\mathcal{T}(y)$ on $\mathbb{T}^{N}$. Moreover, $\mathcal{T}(y)$ is ergodic and in this situation, any function on $\Lambda$ may be regarded as a periodic function on $\mathbb{R}^{N}$ whose period in each coordinate is 1. Therefore, the stochastic-periodic problem \ref{c2tc1})-(\ref{c2tc8}) of Poisson-Nernst-Plank is equivalent to the periodic reiterated problem of Poisson-Nernst-Plank and whose the non-reiterated case is treated in \cite{gagneu1}.
	\end{rem}

	\vspace{0.5cm}
	
	\textbf{Acknowledgments.}	The authors would like to thank the anonymous referee for his/her pertinent remarks, comments and suggestions.
	
	


\end{document}